\documentclass[nohypdvips]{siamart0216}

\usepackage[T1,T2A]{fontenc}
\usepackage[utf8]{inputenc}
\usepackage{amsmath}
\usepackage{amssymb}
\usepackage{graphicx}
\usepackage{algorithm}
\usepackage{algorithmicx}
\usepackage{algpseudocode}
\usepackage{hyperref}
\usepackage{microtype}

\newsiamremark{remark}{Remark}

\DeclareUnicodeCharacter{00A0}{ }

\newcommand{\abs}[1]{\mathopen| #1 \mathclose|}
\newcommand{\dx}{\mathrm d x}
\newcommand{\dw}{\mathrm d w}
\newcommand{\ddz}{\frac{\mathrm d}{\mathrm dz}}

\newcommand{\OO}{\mathcal{O}}
\newcommand{\OOtilde}{\widetilde{\mathcal{O}}}
\newcommand{\MM}{\mathcal{M}}

\begin{document}

\headers{Arbitrary-precision Gauss-Legendre quadrature}{Fredrik Johansson, Marc Mezzarobba}

\title{Fast and rigorous arbitrary-precision computation of Gauss-Legendre quadrature nodes and weights}
\author{Fredrik Johansson\thanks{INRIA -- LFANT, CNRS -- IMB -- UMR 5251, Université de Bordeaux, 33400 Talence, France (\email{fredrik.johansson@gmail.com})}
 \and
Marc Mezzarobba\thanks{Sorbonne Université,  CNRS, Laboratoire d'Informatique
de Paris 6, LIP6, F-75005 Paris, France (\email{marc@mezzarobba.net})}}
\maketitle

\begin{abstract}
We describe a strategy for rigorous
arbitrary-precision evaluation of Legendre polynomials on the unit
interval and its application in the generation of Gauss-Legendre
quadrature rules.
Our focus is on making the evaluation practical for a wide range of
realistic parameters, corresponding to the requirements of numerical
integration to an accuracy of about $100$ to $100\,000$ bits.
Our algorithm combines the summation by rectangular
splitting of several types of expansions in terms of hypergeometric
series with a fixed-point implementation of Bonnet's three-term
recurrence relation.
We then compute rigorous enclosures of the Gauss-Legendre nodes and
weights using the interval Newton method.
We provide rigorous error bounds for all steps of the algorithm.
The approach is validated by an implementation in the Arb library,
which achieves order-of-magnitude speedups over
previous code for computing Gauss-Legendre rules suitable for precisions
in the thousands of bits.
\end{abstract}



\section{Introduction}

The Legendre polynomials $P_n(x)$ are the sequence
of orthogonal polynomials with respect to the unit weight
on the interval $(-1,1)$, normalized so that $P_n(1) = 1$.
Like other classical orthogonal polynomials, Legendre polynomials
satisfy a three-term recurrence, in this case the relation
\begin{equation} \label{eq:recurrence}
  (n + 1) P_{n+1}(x) - (2n + 1) x P_n(x) + n P_{n-1}(x) = 0,
  \quad n \geq 1,
\end{equation}
also known as Bonnet's formula, and a second order differential equation, here
\begin{equation} \label{eq:diffeq}
  (1 - x^2) P''_n(x) - 2 x P'_n(x) + n (n+1) P_n(x) = 0.
\end{equation}

The definition implies that $P_n$ has $n$ roots all located in $(-1,1)$.
Perhaps the most important application of Legendre polynomials
is the Gauss-Legendre quadrature rule
\begin{equation}
\label{eq:glquad}
\int_{-1}^{1} f(x) \dx \approx \sum_{i=0}^{n-1} w_i f(x_i), \qquad
w_i = \frac{2}{(1-x_i^2) [P'_n(x_i)]^2},
\end{equation}
where the \emph{nodes} $x_i$ are the roots of $P_n$.
The quantity~$w_i$ is called the \emph{weight} associated
with the node~$x_i$.

For some applications in computer algebra, number theory,
mathematical physics, and experimental mathematics,
it is necessary to compute integrals to an accuracy of
hundreds of digits, and occasionally even tens of
thousands of
digits~\cite{bailey2011high,BaileyBorweinCrandall2006,Broadhurst2016,Molin2010a}.
The Gauss-Legendre formula~\cref{eq:glquad} achieves an accuracy of~$p$ bits using
$n = \OO(p)$ evaluation points if $f$ is analytic on a neighborhood
of $(-1,1)$, and if the neighborhood
is large (that is, if the path of integration is well isolated
from any singularities of~$f$), then
the constants hidden in the $\OO$
notation are close to the best achievable
by any quadrature rule~\cite{kowalski1985gauss}.
This quality is related to the fact that~\cref{eq:glquad}
maximizes the order of accuracy among $n$-point quadrature rules for integrating polynomials,
being exact when $f$ is any polynomial of degree up to $2n-1$;
as a result, the accuracy is also excellent for analytic integrands
that are well approximated by polynomials.%
\footnote{However, statements about the near-optimality of Gauss-Legendre
quadrature must not be over-interpreted.
Indeed, the rate of convergence of Gauss-Legendre quadrature is \emph{not} optimal asymptotically
when $n \to \infty$ for analytic $f$
on a fixed neighborhood, being improvable by a small factor~\cite{Trefethen2011}.}

In general, the error in~\cref{eq:glquad} can be bounded in terms of
$\sup_{x \in (-1,1)} |f^{(2n)}(x)|$, or, if $f$ is analytic on an
elliptical domain~$D$ with foci at~$\pm 1$, in terms of
$\sup_{z \in D} |f(z)|$ and the semi-axes of the ellipse.
Even when the conditions are not ideal for using \cref{eq:glquad} directly,
rapid convergence is often possible
by combining~\cref{eq:glquad}
with adaptive subdivision of the integration path~\cite{petras2002self}.
We give some elements of comparison between Gauss-Legendre quadrature and
alternative methods,  such as Clenshaw-Curtis quadrature, in
\cref{sec:vsothers}.

The Gauss-Legendre scheme has the drawback
that the quadrature nodes and weights are
somewhat inconvenient to compute.
Indeed, $P_n$ becomes highly oscillatory for
large $n$ and hence presents difficulties
for naive root-finding and polynomial evaluation methods.
The classical Golub-Welsch algorithm avoids accuracy problems
by formulating the task of computing the nodes as
finding the eigenvalues of a tridiagonal matrix~\cite{golub1969calculation},
but this approach is still too slow to be practical for large~$n$.

In the last decade, several authors have contributed to
the development of asymptotic methods
that permit computing any individual node and weight for arbitrarily large $n$
in $\OO(1)$ time, culminating in the 2014 work by
Bogaert~\cite{Bogaert2012,hale2013fast,bogaert2014iteration}.
For a review of this progress, see Townsend~\cite{townsend2015race}.
Of course, the ``$\OO(1)$'' bound assumes that a fixed
level of precision is used. In the prevailing literature
this generally means 53-bit IEEE 754 floating-point arithmetic.
In addition, the available $\OO(1)$ implementations rely in part
on heuristic error estimates without rigorously proved bounds.

The literature on arbitrary precision or rigorous evaluation
is comparatively limited.
Petras \cite{petras1999computation} gave explicit
bounds for the error $|x_k^{(i)} - x_k|$ when
the roots $x_k$ of the Legendre polynomial $P_n$
are approximated using Newton iteration
\begin{equation}
\label{eq:newton}
x^{(i+1)}_k = x^{(i)}_k - \frac{P_n(x^{(i)}_k)}{P'_n(x^{(i)}_k)}
\end{equation}
provided that the initial values $x^{(0)}_k$
are computed by a certain asymptotic formula.
However, Petras did not address the numerical evaluation of $P_n(x)$.
Fousse \cite{fousse2007accurate} discussed the rigorous implementation
of Gauss-Legendre quadrature
using generic polynomial root isolation methods together with
interval Newton iteration
for root refinement,
but did not study fast methods for large $n$.
Code for high-precision Gauss-Legendre quadrature rules can also be found
in packages such as Pari/GP~\cite{PARI2}
and ARPREC~\cite{bailey2002arprec}, but without error analysis
and without special techniques for large~$n$.

In the present article, we are interested in computing Gauss-Legendre
nodes and weights to precisions~$p$ significantly larger than machine
precision---typically in the hundreds to thousands of bits---, with
rigorous error bounds.
As mentioned earlier, certain applications require enclosing the values
of integrals to accuracies of this order,
and it is often reasonable to use quadrature rules of degree~$n$ that
grows roughly linearly with~$p$ for this purpose.
For example, Johansson and Blagouchine~\cite{Johansson2018a}
study the computation of Stieltjes constants to precisions
of hundreds of digits using complex integration, building among
other things on the work described in the present paper.

If we assume that the precision $p$ varies, basic arithmetic operations
are no longer constant-time.
It is well-known that addition, multiplication and division of $p$-bit
numbers (of bounded exponent) can be performed in
$\OOtilde(p)$ operations~\cite{BrentZimmermann2010},
where the notation $\OOtilde(\cdot)$ means that we neglect logarithmic
factors.
It is then clear that any node and weight
can be computed to $p$-bit accuracy in $\OOtilde(n p)$ time,
by performing $\OO(\log p)$ Newton iterations \cref{eq:newton}
from an appropriate initial value.
As a consequence, the full set of nodes and weights
for the degree-$n$ quadrature rule
can be computed in $\OOtilde(n^2 p)$ time.
For numerical integration of analytic functions
where we typically have $p \approx n$, a better (indeed, optimal) estimate
than the classical $\OOtilde(n^3)$ bound is possible.

\begin{theorem}
\label{thm:complexity}
If $p \sim \alpha \, n$ for some fixed~$\alpha$, then the Gauss-Legendre
nodes and weights of degree $n$ can be computed to $p$-bit accuracy in
$\OOtilde(n^2)$ (equivalently, $\OOtilde(p^2)$) bit operations.
\end{theorem}

\begin{proof}[Proof sketch]
Using the formulae in~\cite{petras1999computation},
we can compute good initial values for Newton iteration
in $\OOtilde(n)$ bit operations.
The Newton iteration can be performed for all roots
simultaneously using fast multipoint evaluation, which costs
$\OOtilde(n p)$ bit operations.
Fast multipoint evaluation is numerically unstable and
generically loses $\OO(n)$ bits of accuracy, but we can compensate for
this loss by using $\OO(n)$ guard bits~\cite{kobel2013fast}.
Since $p \sim \alpha n$ by assumption, this does not change the
complexity bound.
\end{proof}

Completing the details
of the proof is a technical exercise.
Despite being elegant in theory,
the algorithm behind \cref{thm:complexity} has a high overhead
in practice, in large part due to the need to work
with greatly increased precision.
Working with expanded polynomials and
processing all roots simultaneously also results in high memory usage and
makes parallelization difficult.
As discussed in \cref{sec:bitburst} below, an approach based on the
``bit-burst'' evaluation method for hypergeometric series
leads to a similar complexity bound and may prove more practical
for extremely large~$p$, but likely not for~$p \leq 10^6$.
We can achieve a slightly worse but still subcubic
complexity of $\OOtilde(n^{5/2})$
by employing
fast multipoint evaluation in a completely different way
to compute $P_n$ values in isolation~\cite{Johansson2014rectangular},
but unfortunately that algorithm also has high overhead.

In this work, we develop rigorous and more practical alternatives to the
asymptotically fast algorithm outlined above.
Our main contribution is to give a complete evaluation strategy
for Legendre polynomials on $[-1,1]$ in ball arithmetic~\cite{vdH:ball,Johansson2017arb}.
Computing the Gauss-Legendre nodes, then,
is a relatively simple application of the results
in~\cite{petras1999computation} together with
the interval Newton method~\cite{moore1979methods}.
For generating Gauss-Legendre quadrature rules with $n \sim \alpha p$,
our algorithm has an asymptotic bit complexity of $\OOtilde(n^3)$ like
classical methods, but much lower overhead.
For parameters $p, n \leq 10^5$ which are most relevant to
applications, the observed running time is effectively subcubic.
Our algorithm outperforms that of \cref{thm:complexity} for
practically any combination of $n$~and~$p$ in that range.
Furthermore, if $p = \OO(1)$, the complexity reduces to $\OOtilde(n)$
as in the machine-precision implementations by
Bogaert~\cite{bogaert2014iteration} and others.

The remainder of this paper is organized as follows.
\Cref{sec:general} gives an overview of our algorithm for evaluating
Legendre polynomials.
This is a hybrid algorithm that switches between different methods,
detailed in the following sections, depending on the values of $n$, $p$,
and~$x$.
In \cref{sec:recurrence}, we prove practical error bounds for
the three-term recurrence~\cref{eq:recurrence}, which can be
efficiently implemented in fixed-point arithmetic.
This method is ideal for $n$ and~$p$ up to a few hundred.
For larger $n$ or $p$, we use a fast method for evaluation
of hypergeometric series.
\Cref{sec:series} discusses the hypergeometric series
expansions that are preferable for different inputs and precision
(including a well-known asymptotic expansion for large $n$),
and \cref{sec:eval} their efficient evaluation.
In \cref{sec:selection}, we propose a strategy to select the
best formula for any combination of~$n, p, x$.
\Cref{sec:bench} presents benchmark results that compare the performance
of our algorithm to some previous implementations as well as the
asymptotically fast algorithm in \cref{thm:complexity}.
Finally, \cref{sec:vsothers} reviews the viability of Gauss-Legendre
quadrature compared to other methods
for extremely high precision integration.

Our code for evaluating Legendre polynomials and computing
Gauss-Legendre nodes and weights is freely available as part of
the Arb library \cite{Johansson2017arb}%
\footnote{\url{http://arblib.org/}}.

\section{General strategy}

\label{sec:general}

We work in the framework of midpoint-radius interval arithmetic,
also called ball arithmetic~\cite{vdH:ball,Johansson2017arb}.
In general, given an integer~$n$ and a ball $x = [m \pm r] = [m-r, m+r]$,
we want to evaluate $P_n(x)$ at~$x$,
yielding an enclosure $y = [m' \pm r']$ such that $P_n(\xi) \in y$
holds for all $\xi \in x$.

We restrict our attention to real~$x \in [-1, 1]$,
which is the most interesting part of the domain for applications.
Since $P_n(-x) = (-1)^n P_n(x)$, we can further
restrict to $0 \le x \le 1$.
Bogaert~\cite{bogaert2014iteration} suggests
working with $P_n(\cos(\theta))$ instead of $P_n(x)$
directly to improve numerical stability for $x$ close to~$\pm 1$.
This is not necessary in arbitrary-precision arithmetic since
a slight precision increase (of the order of $\OO(\log n)$ bits, since
the distance between two successive roots of~$P_n$ close to~$\pm1$ is
about $1/n^2$) works as well.

We note that, for rigorous evaluation of $P_n(z)$ with complex $z$
as well as Legendre functions of non-integer order $n$,
generic methods for the hypergeometric ${}_2F_1$ function
are applicable if $n$ is not extremely large; see~\cite{johansson2016hypergeometric}.
Real $|x| > 1$ can also be handled easily using naive methods.

In view of the use of Newton's method to compute the roots,
we also need to evaluate the derivative $P'_n(x)$,
typically at the same time as $P_n(x)$ itself.
A simple option is to deduce $P'_n(x)$ from
$P_n(x)$ and $P_{n-1}(x)$ using
\begin{equation} \label{eq:mixed}
  (x^2 - 1) P'_n(x) = n \bigl( x P_n(x) - P_{n-1}(x) \bigl).
\end{equation}
When $x$ is close to~$\pm 1$, though, this formula involves a
cancellation of about $\abs{\log_2(1 - x)}$ bits in the subtraction,
followed by a division by $x^2 - 1$.
Therefore, a direct evaluation of $P'_n(x)$ may be preferable to reduce
the working precision.
We use either of these strategies depending on the values of
$n$, $p$, and~$x$.

Our evaluation algorithms rely on ball arithmetic internally to
propagate the error bounds up to the final result.
Therefore, to ensure that the enclosure output by our evaluation
algorithm contains the image of the input, it is enough to have bounds
on the truncation errors associated to each of the approximate
expressions of~$P_n$ that we use.
The corresponding bounds are stated in equations \cref{eq:truncerr0},
\cref{eq:truncerr1}, and \cref{eq:truncerr2}.

To limit overestimation and computational overhead, we deviate from the
direct use of ball arithmetic on two occasions.
First, \cref{alg:gmprec} does not keep track of round-off
errors internally: instead, we prove an a priori bound on the
accumulated error (\cref{cor:bound-gmprec}) and add it to the
radius of the output ball after calling that algorithm.
Second, since some methods would produce unsatisfactorily large
enclosures when executed on input balls $x = [m \pm r]$ of radius
$r > 0$, we evaluate $P_n(m)$ (with higher internal precision if
necessary) and $P'_n(x)$ and use a first-order bound
\[ \max_{\xi \in x} |P_n(\xi) - P_n(m)|
   \le r \max_{\xi \in x} |P'_n(\xi)| \]
to separately bound the propagated error.
Similarly, we use a bound for $P''_n$ to compute a reasonably
tight enclosure for $P'_n([m \pm r])$.
Suitable bounds are given in \cref{prop:prop-bound}.
Denote by $[z^n] f(z)$ the coefficient of index $n$ in a power
series $f(z)$, and write $f(z) \ll_z \hat f(z)$
for two power series $f, \hat f$ if
$\hat f$ has nonnegative coefficients and
$| [z^n] f(z) | \leq [z^n] \hat f(z)$ for all $n$.

\begin{lemma} \label{lemma:majorants}
If $f$, $g$, $\hat f$, $\hat g$ are such that
$f(z) \ll_z \hat f(z)$ and $g(z) \ll_z \hat g(z)$, then
$\int_0^z f \ll_z \int_0^z \hat f$ and
$f(z)g(z) \ll_z \hat f(z) \hat g(z)$.
\end{lemma}

\begin{proposition} \label{prop:prop-bound}
The following bounds hold for $-1 \leq x \leq 1$:
\begin{align}
\label{eq:prop-bound}
  |P_n'(x)| &\le \min\left(
      \frac{2^{3/2}}{\sqrt{\pi}} \frac{\sqrt n}{(1-x^2)^{3/4}},
      \frac{n(n+1)}{2}
  \right), \\
\label{eq:prop-bound1}
  |P_n''(x)|
  &\le \min\left(
      \frac{2^{5/2}}{\sqrt{\pi}} \frac{n^{3/2}}{(1-x^2)^{5/4}},
      \frac{(n-1) n (n+1) (n+2)}{8}
  \right).
\end{align}
\end{proposition}

\begin{proof}
It is classical that Legendre polynomials are given by the generating series
\begin{equation} \label{eq:generating-series}
  F(x, z)
  = \sum_{n=0}^{\infty} P_n(x) z^n
  = \frac{1}{\sqrt{1 - 2 x z + z^2}}.
\end{equation}
Differentiation with respect to~$x$ yields
\[
  \sum_{n=0}^{\infty} P'_n(x) z^n
  = \frac{z F(x, z)}{1 - 2 x z + z^2},
  \qquad
  \sum_{n=0}^{\infty} P''_n(x) z^n
  = \frac{3 z^2 F(x, z)}{(1 - 2 x z + z^2)^2}.
\]
Set $\theta = \arccos x$, so that the roots of $1 - 2 x z + z^2$
are $e^{\pm i \theta}$.
Then, in the notation of \cref{lemma:majorants}, we have the bound
\[ \frac{1}{1 - 2 x z + z^2}
   = \frac{1}{2 i \sin \theta} \left( \frac{1}{z - e^{i \theta}} -
                                      \frac{1}{z - e^{-i \theta}} \right)
   = \sum_{n=0}^{\infty} \frac{\sin\bigl((n+1) \theta\bigr)}{\sin(\theta)} z^n
   \ll_z \frac{\sin(\theta)^{-1}}{1 - z}. \]
In addition, Bernstein's inequality for the Legendre polynomials
\cite{AntonovHolsevnikov1980}
(see also~\cite{ChowGatteschiWong1994})
combined with the logarithmic convexity of the Gamma function yields
\[ |P_n(x)| \leq \frac{\sqrt{2}}{\sqrt{\pi}} \frac{1}{\sqrt{\sin \theta}}
                 \frac{1}{\sqrt{n+1/2}}
            \leq \frac{\sqrt 2}{\sqrt{\pi}} \frac{1}{\sqrt{\sin \theta}}
                 \frac{\Gamma(n+1/2)}{\Gamma(n+1)},
\]
and hence
\[ F(x, z) \ll_z \frac{\sqrt{2}}{\sqrt{\sin \theta}}
               \sum_{n=0}^{\infty} \frac{1}{\sqrt{\pi}}
                                  \frac{\Gamma(n+1/2)}{\Gamma(n+1)} z^n
            = \frac{\sqrt{2}}{\sqrt{\sin \theta}}
              \frac{1}{\sqrt{1-z}}. \]
By \cref{lemma:majorants}, these bounds combine into
\[ \frac{\mathrm d F}{\mathrm d x}
   \ll_z \frac{\sqrt{2}}{\sin(\theta)^{3/2}} \frac{1}{(1-z)^{3/2}},
   \qquad
   \frac{\mathrm d^2 F}{\mathrm d x^2}
   \ll_z \frac{\sqrt{2}}{\sin(\theta)^{5/2}} \frac{3 z^2}{(1-z)^{5/2}}. \]
Since
$[z^n] (1 - z)^{- k - 1/2} = \Gamma(n + 1/2)/(\Gamma(k+1/2) \Gamma(n - k + 1))$
and using again the logarithmic convexity of~$\Gamma$, we conclude
that
\[ |P'_n(x)| \leq \frac{\sqrt{2}}{\sin(\theta)^{3/2}}
                  \frac{2}{\sqrt{\pi}} \frac{\Gamma(n+1/2)}{\Gamma(n)}
             \leq \frac{2^{3/2}}{\sqrt{\pi}} \frac{\sqrt n}{\sin(\theta)^{3/2}},
   \qquad
   |P''_n(x)| \leq \frac{2^{5/2}}{\sqrt{\pi}} \frac{n^{3/2}}{(1-x^2)^{5/4}}.
\]
The result follows since all derivatives of Legendre polynomials reach
their maximum at~$x=1$ (or by using the bounds
$(z\!-\!e^{\pm i \theta})^{-1}, F(x, z) \ll_z (1\!-\!z)^{-1}$
and \cref{lemma:majorants}).
\end{proof}

\begin{remark}
By the same reasoning, the inequality
\[ |P^{(k)}_n(x)| \leq \frac{2^{k+1/2}}{\sqrt{\pi}}
                       \frac{n^{k-1/2}}{(1 - x^2)^{(2n+1)/4}} \]
actually holds for all~$k$. Unfortunately, it seems to
overestimate the envelope of~$|P_n^{(k)}|$ by a factor about~$2^k$
in the region where it is smaller than~$P_n^{(k)}(1)$.
\end{remark}

Based on these reductions, we assume from now on that $x$~is a
floating-point number with $0 \leq x \leq 1$.
Our main algorithm for evaluating~$P_n$ at~$x$ combines the following
methods:
\begin{itemize}
  \item the iterative computation of~$P_n(x)$
  via the three-term recurrence~\cref{eq:recurrence},
  \item the asymptotic expansion~\cref{eq:asymptotic}
  of $P_n(x)$ as~$n \to \infty$,
  \item the usual expanded expression \cref{eq:pzeroseries0},
  \cref{eq:pzeroseries1} of $P_n$ in the monomial basis,
  \item the analogous terminating expansion \cref{eq:poneseries}
  at~$1$.
\end{itemize}
All three expansions can be written as hypergeometric
series, i.e., sums of the form $\sum_k c_k \xi^k$ where $c_k/c_{k-1}$
is a rational function of~$k$.

The constraints and heuristics used to select between these methods
are described in detail below.
Roughly speaking,
the three-term recurrence is used for small index~$n$ and precision~$p$, when~$x$
is not too close to~$1$;
the asymptotic series when $n$~is large enough, again with $x$~not too
close to~$1$;
the expansion at~$0$ for large $p$ unless $x$ is close to~$1$;
and finally the expansion at~$1$ in the remaining cases
when $x$ is close to~$1$.

For an evaluation at $p$-bit precision, we choose parameters such as
truncation orders and internal working precision to target an absolute
error of $2^{-p-p'}$ for some $p' = \OO(\log n)$, corresponding to
a relative error of about $2^{-p}$ measured with respect
to monotone envelopes for $P_n(x)$ and $P'_n(x)$ as in~\cite{Bogaert2012}.
The relative error of a computed ball for $P_n(x)$ where $x$ is near a zero $x_k$
can be arbitrarily large, but the relative error of $P'_n(x)$ near $x_k$
will be small, which is sufficient for the Newton iteration method.
Since the output consists of a ball, we also have the option of catching
a result with large relative error and repeating the evaluation with a
higher precision as needed.

\section{Basecase recurrence}

\label{sec:recurrence}

For small $n$, a straightforward way to compute~$P_n(x)$ is to apply
the three-term recurrence \cref{eq:recurrence},
starting from $P_0(x)=1$ and ${P_1(x) = x}$.
Computing~$P_n(x)$ by this method takes about $(\MM(t) + \OO(t))\, n$
bit operations, where $t$ is the working precision and $\MM(t)$ denotes the
cost of $t$-bit multiplication.
It is thus attractive for small $n$ and $t$, especially when both
$P_n(x)$ and $P'_n(x)$ are needed, since we can get $P_{n-1}(x)$ at no
additional cost.

Fix $x \in [-1, 1]$, and let $p_n = P_n(x)$.
Bonnet's formula~\cref{eq:recurrence} gives
\begin{equation} \label{eq:rec-bis}
  p_{n + 1} =
    \frac{1}{n+1}
    \bigl( (2n +1) x p_n - n p_{n-1} \bigr),
  \qquad n \geq 0.
\end{equation}
In a direct implementation of this recurrence in ball arithmetic, the
width of the enclosures would roughly double at every iteration,
requiring to increase the internal working precision by $\OO(n)$ bits.
We avoid this issue by performing an a priori round-off error
analysis (to be presented now) of the evaluation that yields a less
pessimistic bound on the accumulated error.
Additionally, the static error bound allows us to implement the
recurrence in fixed-point arithmetic, avoiding the overhead of
floating-point and interval operations.

Suppose $x = \hat x \, 2^{-t}$ with $\hat x \in \mathbb Z$
is a given fixed-point number.
Let $\lceil u \rfloor$ denote the integer truncation of a real
number~$u$ (note that this is not the same thing as rounding to the nearest
integer; however, any rounding function would do).
The integer sequence $(\hat p_n)$ defined by
\begin{equation} \label{eq:rec-fxpt}
  \hat{p}_0 = 2^t, \qquad
  \hat{p}_1 = \hat x, \qquad
  \hat{p}_{n + 1} =
    \left\lceil \frac{1}{n + 1}  \bigl( (2 n + 1) \lceil \hat{x}
      \hat{p}_n 2^{- t} \rfloor - n \hat{p}_{n - 1} \bigr)
    \right\rfloor
\end{equation}
is easy to compute using only integer arithmetic, and
$\hat p_n \, 2^{-t}$ is an approximation of~$p_n$.
\Cref{alg:gmprec} provides a complete C implementation
using GMP~\cite{granlund2017}.
As a small optimization, we delay the division by~$n+1$ until we have
accumulated a denominator of the size of a machine word.

\begin{algorithm}[h]
  \caption{Evaluation of Legendre polynomials in GMP fixed-point arithmetic}
  \small
  \label{alg:gmprec}
  \begin{algorithmic}[1]
    \Require An integer $x$ and $t \ge 0$ such that $|2^{-t} x| \le 1$, and $n \ge 1$
    \Ensure $p, q$ such that $|2^{-t} p - P_{n-1}(2^{-t} x)|, |2^{-t} q - P_{n}(2^{-t} x)| \le (0.75 \, (n+1)(n+2) + 1) \, 2^{-t}$
\State \verb!void legendre(mpz_t p, mpz_t q, int n, const mpz_t x, int t) {!
\State \verb!    mpz_t tmp; int k; mpz_init(tmp);!
  \Comment{Comments use the notation of}
\State \verb!    mp_limb_t denlo, den = 1;!
  \Comment{\rlap{the proof of \cref{cor:bound-gmprec}}%
           \phantom{Comments use the notation of}}
\State \verb!    mpz_set_ui(p, 1); mpz_mul_2exp(p, p, t);!
  \Comment{$\mathtt p_0 = 2^t$}
\State \verb!    mpz_set(q, x);!
  \Comment{$\mathtt q_0 = \hat x$}
\State \verb!    for (k = 1; k < n; k++) {!
\State \verb!        mpz_mul(tmp, q, x); mpz_tdiv_q_2exp(tmp, tmp, t);!
  \Comment{$\lceil \hat x \, \mathtt q_{k-1} \, 2^{-t} \rfloor$}
\State \verb!        mpz_mul_si(p, p, -k*k)!
\State \verb!        mpz_addmul_ui(p, tmp, 2*k+1);!
  \Comment{$- k^2 \mathtt p_{k-1} + (2k+1) \, \mathtt{tmp}$}
\State \verb!        mpz_swap(p, q);!
\State \verb!        if (mpn_mul_1(&denlo, &den, 1, k+1)) {!
  \Comment{If multiplication overflows}
\State \verb!            mpz_tdiv_q_ui(p, p, den);!
  \Comment{$\lceil \mathtt p / \mathtt d_{k-1} \rfloor$}
\State \verb!            mpz_tdiv_q_ui(q, q, den);!
\State \verb!            den = k+1;!
  \Comment{$\mathtt d_{k} = k+1$}
\State \verb!        } else den = denlo;!
  \Comment{$\mathtt d_{k} = (k+1) \, \mathtt d_{k-1}$}
\State \verb!    }!
\State \verb!    mpz_tdiv_q_ui(p, p, den/n); mpz_tdiv_q_ui(q, q, den);!
\State \verb!    mpz_clear(tmp);!
\State \verb!}!
  \end{algorithmic}
\end{algorithm}

To bound the difference $\abs{\hat p_n \, 2^{-t} - p_n}$, we analyze
the effect on the result of a small perturbation in each iteration
of \cref{eq:rec-bis}.
The bound is based on a classical linearity argument (compare, e.g.,
\cite{Wimp1984}) combined with majorant series techniques.

\begin{proposition} \label{prop:rec-error}
Suppose that a real sequence $(\tilde p_n)_{n \geq -1}$ satisfies
$\tilde p_0 = 1$ and
\begin{equation} \label{eq:rec-pert}
  \tilde{p}_{n + 1} =
    \frac{1}{n+1}
    \bigl( (2n +1) x \tilde{p}_n - n \tilde{p}_{n-1} \bigr)
    + \varepsilon_n,
  \qquad n \geq 0.
\end{equation}
for arbitrary real numbers $\varepsilon_n$
with $\abs{\varepsilon_n} \leq \bar\varepsilon$ for all $n$.
Then we have
\[
  \abs{\tilde p_n  - P_n(x)}
  \leq \frac{(n+1)(n+2)}{4} \bar \varepsilon
\]
for all $n \geq 0$.
\end{proposition}

\begin{proof}
Let
$\delta_n = \tilde{p}_n - p_n$
and
$\eta_n = (n + 1) \varepsilon_n$.
Subtracting \cref{eq:rec-bis} from \cref{eq:rec-pert} gives
\begin{equation} \label{eq:rec-error}
  (n + 1) \delta_{n + 1}
  = (2 n + 1) x \delta_n - n \delta_{n - 1} + \eta_n,
\end{equation}
with $\delta_0 = 0$.
Consider the formal generating series
$\delta(z) = \sum_{n \geq 0} \delta_n z^n$
and
$\eta(z) = \sum_{n \geq 0} \eta_n z^n$.
Noting that \cref{eq:rec-error} holds for all $n \in \mathbb{Z}$ if
the sequences $(\delta_n)$ and $(\eta_n)$ are extended by~$0$ for
$n < 0$
and using the relations
\[
  \sum_{n=-\infty}^{\infty} f_{n - 1} z^n
  = z \sum_{n=-\infty}^{\infty} f_n z^n,
  \qquad
  \sum_{n=-\infty}^{\infty} n f_{n} z^n
  = z \ddz \sum_{n=-\infty}^{\infty} f_n z^n,
\]
we see that \cref{eq:rec-error} translates into
\[ (1 - 2 xz + z^2) z \ddz \delta (z)
   = z (x - z) \delta (z) + z \eta (z). \]
The solution of this differential equation with $\delta (0) = 0$ reads,
cf.~\cref{eq:generating-series},
\[ \delta(z) = p(z)  \int_0^z \eta(w) \, p(w) \, \dw,
   \qquad
   p(z) = \sum_{n=0}^{\infty} p_n z^n
        = F(x, z)
        = \frac{1}{\sqrt{1 - 2 xz + z^2}}. \]
This is an exact expression of the ``global'' error $\delta$ in terms
of the ``local'' errors~$\varepsilon_n$.
Since $\abs{p_n} = \abs{P_n(x)} \leq 1$
and $\abs{\eta_n} \leq (n+1) \bar \varepsilon$,
it follows by \cref{lemma:majorants} that
\[
  \abs{\delta_n}
  = \left| [z^n] \left( p(z)  \int_0^z \eta(w) \, p(w) \, \dw \right) \right|
  \leq [z^n] \left(
    \frac{1}{1-z}
    \int_0^z \frac{\bar \varepsilon}{(1 - w)^2} \frac{1}{1-w} \dw \right)
\]
and therefore
\[
  \abs{\delta_n}
  \leq [z^n] \left( \frac12 \frac{\bar \varepsilon}{(1-z)^3} \right)
  = \frac{(n+1)(n+2)}{4} \, \bar \varepsilon. \qed
\]
\end{proof}

\begin{corollary}
\label{cor:bound-gmprec}
Suppose that $x = \hat x \, 2^{-t}$ for some $t \geq 0$ and
$\hat x \in \mathbb Z$.
The sequence $(\hat p_n)_{n \geq 0}$ defined by \cref{eq:rec-fxpt}
satisfies
\begin{equation} \label{eq:bound-fxprec}
  \abs{ \hat p_n 2^{-t} - p_n }
  \leq 0.75 \, (n+1) (n+1) \, 2^{-t},
  \qquad
  n \geq 0.
\end{equation}
Furthermore, the quantities $p$, $q$ returned by \cref{alg:gmprec}
are such that
\begin{equation} \label{eq:bound-gmprec}
  |p - 2^t P_{n-1}(x)|, |q - 2^t P_{n}(x)|
  \le 0.75 (n+1)(n+2) + 1.
\end{equation}
\end{corollary}

\begin{proof}
We can write
\[
  \hat{p}_{n + 1} = \frac{1}{n + 1}  \bigl((2 n + 1)  (\hat{x} \,
        \hat{p}_n \, 2^{- t} + \alpha_n) - n \, \hat{p}_{n - 1}\bigr)
        + \beta_n
\]
for some $\alpha_n$, $\beta_n$ of absolute value at most one, and
hence
\[
  \hat{p}_{n + 1} = \frac{1}{n + 1}  \bigl((2 n + 1) \, \hat{x} \,
  \hat{p}_n \, 2^{- t} - n \, \hat{p}_{n - 1}\bigr) + \varepsilon_n,
  \qquad
  \varepsilon_n = \frac{2 n + 1}{n + 1} \alpha_n + \beta_n.
\]
where $\abs{\varepsilon_n} \leq 3$.
\cref{prop:rec-error} applied to
$\tilde p_n = \hat p_n \, 2^{-t}$
then provides the bound~\cref{eq:bound-fxprec}.

\newcommand{\vp}{\mathtt p}
\newcommand{\vq}{\mathtt q}
\newcommand{\vd}{\mathtt d}

Turning to \cref{alg:gmprec},
let $\vp_0$, $\vq_0$, $\vd_0$ denote the values of the variables
$\vp$, $\vq$, $\mathtt{den}$ before the loop,
and $\vp_k$, $\vq_k$, $\vd_k$ their values at the end of iteration~$k$.
Consider the sequence
$\tilde p_k = 2^{-t} \vq_{k-1}/\vd_{k-1}$, $k \geq 1$,
extended by $\tilde p_0 = 1$ and an arbitrary (finite)~$\tilde p_{-1}$.
For all $k \geq 1$, depending whether the conditional branch is taken, we
have one of the systems of equations
\begin{align}
  \vp_k &= \vq_{k-1}, &
  \vq_k &= (2k+1) \lceil \hat x \vq_{k-1} 2^{-t} \rfloor
          - k^2 \vp_{k-1}, &
  \vd_k &= (k+1) \, \vd_{k-1} \\
  \vp_k &= \left\lceil \frac{\vq_{k-1}}{\vd_{k-1}} \right\rfloor, &
  \vq_k &= \left\lceil \frac{
          (2k+1) \lceil \hat x \vq_{k-1} 2^{-t} \rfloor - k^2 \vp_{k-1}}
          {\vd_{k-1}} \right\rfloor, &
  \vd_k &= k+1.
\end{align}
In both cases, we can write
\begin{align*}
  \frac{\vp_k}{\vd_k}
  &= \frac{\vq_{k-1}}{(k+1) \vd_{k-1}} + \frac{\alpha_k}{k+1},
   &
  \frac{\vq_k}{\vd_k}
  &= \frac{ (2k+1) (x \vq_{k-1} + \beta_k) - k^2 \vp_{k-1} }
          { (k+1) \vd_{k-1} }
     + \frac{\gamma_k}{k+1}
\end{align*}
with $|\alpha_k|, |\beta_k|, |\gamma_k| \leq 1$.
The first equation implies
$2^{-t} k \, \vp_{k-1}/\vd_{k-1} = \tilde p_{k-1} + \alpha_{k-1} 2^{-t}$
for $k \geq 2$.
Since the latter equality also holds for $k = 1$ with $\alpha_0 = 0$, we can
substitute it in the second equation, yielding
\[
  \tilde p_{k+1}
  = \frac{ (2k+1) x \tilde p_k - k \tilde p_{k-1} }{k+1}
     + 2^{-t} \left(
          \frac{-k \alpha_{k-1}}{k+1}
        + \frac{(2k + 1) \beta_k}{(k+1) \vd_{k-1}}
        + \frac{\gamma_k}{k+1}
       \right).
\]
This relation holds for $k \geq 1$, and we extend it to $k=0$ by setting
$\beta_0 = \gamma_0 = 0$.
Thus, $\tilde p_k$ also satisfies~\cref{eq:rec-pert} with
$|\varepsilon_n| \leq 3 \cdot 2^{-t}$, and \cref{prop:rec-error} applies.
The final values of $\vq$~and~$\vp$ are respectively
$\lceil 2^t \tilde p_n \rfloor$
and
$\lceil n \vp_{n-1}/\vd_{n-1} \rfloor
= \lceil 2^t \tilde p_{n-1} + \alpha_{n-1} \rfloor$,
whence the bound~\cref{eq:bound-gmprec}.
\end{proof}

We do not use asymptotically faster evaluation techniques for large~$n$ in
combination with this recurrence, since the series expansions to be
presented next perform very well in this case.

\section{Series expansions}

\label{sec:series}

For large $n$ or $p$, we employ series expansions of $P_n(x)$ with
respect to either $n$ or $x$ rather than the algorithm from the
previous section.
The coefficients of the series are also computed by recurrence,
but fewer than $n$ terms will typically be required.
Let us now review the various series expansions that we are using (an
asymptotic expansion as $n \to \infty$, series expansions at $x=0$
and $x=1$), before discussing their efficient evaluation in the next
section.

\subsection{Asymptotic series}

\label{sec:series-asy}

For fixed $|x| < 1$ or equivalently $x = \cos(\theta)$ with $0 < \theta < \pi$,
an asymptotic expansion for $P_n(x)$ as $n \to \infty$
can be given as \cite[Eq.~3.4]{Bogaert2012}
\begin{equation}
\label{eq:asymptotic}
P_n(\cos(\theta)) = \left(\frac{2}{\pi \sin(\theta)}\right)^{1/2}
\sum_{k=0}^{K-1} C_{n,k} \frac{\cos(\alpha_{n,k}(\theta))}{\sin^k(\theta)}
+ \xi_{n,K}(\theta)
\end{equation}
(for arbitrary $K \geq 1$), where
\begin{equation}
\label{eq:Cnk-Gamma}
C_{n,k} = \frac{[\Gamma(k+1/2)]^2 \Gamma(n+1)}{\pi 2^k \Gamma(n+k+3/2) \Gamma(k+1)},
\end{equation}
\begin{equation*}
\alpha_{n,k}(\theta) = (n+k+1/2) \theta - (k+1/2) \pi / 2,
\end{equation*}
and the error term satisfies
\begin{equation}
\label{eq:truncerr0}
|\xi_{n,K}(\theta)| < 2 \left(\frac{2}{\pi \sin(\theta)}\right)^{1/2} \frac{C_{n,K}}{\sin^K(\theta)}.
\end{equation}
The coefficients $C_{n,k}$ are a hypergeometric sequence with
\begin{equation}
\label{eq:asymprec}
\frac{C_{n,k}}{C_{n,k-1}} = \frac{(2k-1)^2}{4 k (2n+2k+1)}, \quad C_{n,0} = \frac{1}{\sqrt{\pi}} \frac{4^n}{(n+\tfrac{1}{2}) {2n \choose n}}.
\end{equation}
To evaluate the error bound, we can use the following inequality deduced
from~\cref{eq:Cnk-Gamma}:
\begin{equation}
\label{eq:truncerr0b}
C_{n,k} \le \frac{1}{\pi n^{1/2}} \frac{k! n!}{2^k (n+k)!} \le \frac{1}{\pi n^{1/2}} \frac{k!}{(2n)^k},
\quad n, k \ge 1.
\end{equation}

The condition
$|\xi_{n,K}(\theta)| \leq 2^{-p}$
is satisfiable as soon as
$\sin(\theta) \geq (p+3)\ln(2)/(2n)$,
as we can see by choosing
$K = \lfloor 2 n \sin(\theta) \rfloor$
and using the inequality
$k! \leq k^{k+1/2} e^{1-k}$.
Since the Gauss-Legendre nodes are distributed linearly in~$\theta$,
the asymptotic expansion gives a candidate algorithm for all but about
$(\log(2)/\pi)\,p/n$ of the nodes as $p/n \to 0$.
It is in fact a convergent series
when $2 \sin(\theta) > 1$, which allows evaluating $P_n(x)$ to unbounded
accuracy for fixed $n$ when $\tfrac{1}{6}\pi < \theta < \tfrac{5}{6} \pi$.
The particular form \cref{eq:asymptotic}
must be used for this purpose; there is a slightly different version of the expansion
which is asymptotic to $P_n(x)$ (for fixed $K$ when $n \to \infty$)
but paradoxically
converges to $2 P_n(x)$ (for fixed~$n$ when $K \to \infty$); see \cite[Section~10.3]{Olver1997} and \cite[Section~18.15(iii)]{Olver2010}.

We can restate \cref{eq:asymptotic} as a hypergeometric series
by working with complex numbers.
Letting $\omega = 1 - (x/y) i$, with $x = \cos(\theta)$ and $y =
\sin(\theta)$ as usual, we have
\[ (1-i) (x+iy)^{n+1/2} \omega^k
= \sqrt2 \, e^{-i \frac{\pi}{4}} \, e^{i (n+1/2)\theta} \,
  e^{i k (\theta-\pi/2)} \, y^{-k}
= \sqrt2 \, \frac{\exp(i \, \alpha_{n,k}(\theta))}{\sin^k(\theta)}, \]
and hence
\begin{equation}
\label{eq:asymptoticcomplex}
P_n(x) = \frac{1}{\sqrt{\pi y}} \, \operatorname{Re}\left[
(1-i) (x+y i)^{n+1/2}
\sum_{k=0}^{K-1} C_{n,k} \omega^k\right] + \xi_{n,K}(\theta).
\end{equation}
This eliminates the explicit trigonometric functions and permits using
\cref{alg:hyprs} below to evaluate the series.

The evaluation of \cref{eq:asymptoticcomplex} in ball arithmetic
is numerically stable, and we therefore only need to add a few
guard bits to the working precision.

\subsection{Expansion at zero}

\label{sec:series-zero}

If $n = 2d$ is even, the expansion of $P_n(x)$ in the monomial basis
reads
\begin{align}
\label{eq:pzeroseries0}
\begin{split}
P_{2d}(x) &= (-1)^d \sum_{k=0}^d \frac{(-1)^k}{2^n} {n \choose d-k} {n+2k \choose n} x^{2k} \\
&= \frac{(-1)^d}{2^{2d}} {2d \choose d} \sum_{k=0}^d A_{-1}(d,k) (-x^2)^k,
\end{split}
\end{align}
and if $n = 2d+1$ is odd, we have
\begin{align}
\label{eq:pzeroseries1}
\begin{split}
P_{2d+1}(x) &= (-1)^d x \sum_{k=0}^d \frac{(-1)^k}{2^n} {n \choose d-k} {n+2k+1 \choose n} x^{2k} \\
&= \frac{(-1)^d (d+1)}{2^{2d+1}} {2d+2 \choose d+1} \,x\, \sum_{k=0}^d A_{+1}(d,k) (-x^2)^k
\end{split}
\end{align}
where the hypergeometric sequences $A_{\pm1}$ can be defined by
$A_{\pm 1}(d,0) = 1$ and
\begin{equation}
\label{eq:pzerorec}
\frac{A_{\sigma}(d,k)}{A_{\sigma}(d,k-1)} = \frac{(d-k+1) (2d+2k+\sigma)}{k (2k+\sigma)}, \quad \sigma \in \{-1,+1\}.
\end{equation}

At very high precision, we evaluate the full polynomials,
where \cref{eq:pzeroseries0} and \cref{eq:pzeroseries1}
have the advantage compared to other expansions
of requiring only $n/2$ terms due to the odd-even form.
At lower precision $p$, the high order terms will be smaller than
$2^{-p}$ when $|x|$ is small, and we can truncate the series
accordingly and add a bound for the omitted terms
to the radius of the computed ball.
When the series are truncated after the $k = K - 1$ term
(for any $K < d + 1$), comparison with a geometric series
shows that the error is bounded by the first
omitted term times a simple factor.

\begin{proposition}
For $\sigma \in \{-1,+1\}$, the error when truncating
the bottom sum in
\cref{eq:pzeroseries0} or \cref{eq:pzeroseries1} (with prefactors removed)
after the $k = K - 1$ term satisfies
\begin{equation}
\label{eq:truncerr1}
\left| \sum_{k=K}^d A_{\sigma}(d,k) (-x^2)^k \right| \le \frac{A_{\sigma}(d,K) |x|^{2K}}{1-\alpha},
\quad\!\! \alpha = |x|^2 \frac{(d-K+1)(2d+2K+\sigma)}{K (2K+\sigma)}
\end{equation}
provided that $\alpha < 1$.
\end{proposition}

For bounding $A_{\sigma}(d,K)$ in this expression, and for selecting an
appropriate truncation point $K$, we use the binomial closed forms
\cref{eq:pzeroseries0}, \cref{eq:pzeroseries1} together with the
remarks in \cref{sec:binomial}.

The alternating series \cref{eq:pzeroseries0} and \cref{eq:pzeroseries1}
may suffer from significant cancellation, which requires use of
increased precision.
We can estimate the magnitude by
noting that no cancellation occurs
if $x$ is an imaginary number.
Solving the majorizing recurrence $f_n = 2 |z| f_{n-1} + f_{n-2}$ with
$f_0 = 1$, $f_1 = |z|$ shows that
\[ |P_n(z)| \le |P_n(i|z|)| \le \left(|z| + \sqrt{1 + |z|^2}\right)^n. \]
Therefore, the possible cancellation assuming that $|P_n(x)| \approx 1$
is about
\begin{equation*}
p_A = n \log_2\bigl(|x| + \sqrt{1 + |x|^2}\bigr)
\end{equation*}
bits (which is at most $n \log_2 (1+\sqrt{2}) \approx 1.27\, n$),
so using ball arithmetic with about $p + p_A$ bits of working precision
for the series evaluation gives $p$-bit accuracy.

\subsection{Expansion at one}

\label{sec:series-one}

Expanding at $x = 1$ yields
\begin{equation}
\label{eq:poneseries}
P_n(x) = \sum_{k=0}^n c_{n,k} u^k, \quad c_{n,k} = {n \choose k} {n+k \choose k}
\end{equation}
where $u = (x-1)/2$.
The coefficients $c_{n,k}$ are hypergeometric with initial value $c_{n,0} = 1$
and term ratio
\begin{equation}
\label{eq:poneprec}
\frac{c_{n,k}}{c_{n,k-1}} = \frac{(n-k+1)(n+k)}{k^2}.
\end{equation}
As in the previous section, we can truncate \cref{eq:poneseries} and
bound the error by comparison with a geometric series.

\begin{proposition}
\label{prop:poneseries}
The error when truncating \cref{eq:poneseries} after the $k = K - 1$ term satisfies
\begin{equation}
\label{eq:truncerr2}
\left| \sum_{k=K}^n c_{n,k} u^k \right| \le \frac{c_{n,K} |u|^K}{1-\alpha},
\qquad \alpha = |u| \frac{(n-K)(n+K+1)}{(K+1)^2}
\end{equation}
provided that $\alpha < 1$.
\end{proposition}

For $u \ge 0$, \cref{eq:poneseries} does not suffer from cancellation.
For $u < 0$, we can estimate the amount of
cancellation from the magnitude of $P_n(x')$ where $|u| = (x'-1)/2$.
For not too large $x' \ge 1$, a very good approximation is
\[ P_n(x') \leq 2 \sum_{k=0}^{\infty} \frac{n^{2k}}{k!^2} |u|^k
           = 2\,I_0(2n\sqrt{|u|})
           \leq 2\,e^{2n\sqrt{|u|}}. \]
We therefore need about $2n\sqrt{\max(0,-u)} / \ln(2)$ bits
of increased precision.

We can compute $P'_n$ from $P_n$ and $P_{n-1}$, but since this
involves a division by $1-x^2$, it is better to
evaluate $P'_n$ directly when $x$ is close to 1.
We have
$P'_n(x) = \sum_{k=0}^{n-1} c'_{n,k} u^k$
where $c'_{n,k} = (k+1) c_{n,k+1} / 2$ satisfies
\begin{equation}
\label{eq:ponerecprime}
  c'_{n,0} = \frac{n(n+1)}{2}, \qquad
  \frac{c'_{n,k}}{c'_{n,k-1}} = \frac{(n-k)(n+k+1)}{k(k+1)}.
\end{equation}
Since $c'_{n,k} \le n c_{n,k+1}$, the analog
\begin{equation}
\label{eq:truncerr2b}
\left| \sum_{k=K}^n c'_{n,k} u^k \right| \le n {n \choose K+1}{n+K+1 \choose K+1} |u|^K \frac{1}{1-\alpha}
\end{equation}
of \cref{prop:poneseries} holds with $u$ and $\alpha$ as above.

\section{Fast evaluation of series expansions}

\label{sec:eval}

All these series expansions are amenable to fast evaluation techniques
specific to multiple-precision arithmetic. Using such fast summation
algorithms is critical for achieving good performance at high precision.
We now discuss the algorithm that we use for evaluating the series of
the previous section.

\subsection{Rectangular splitting}

\label{sec:serieseval}

We use rectangular splitting~\cite{Smith1989,Johansson2014rectangular}
to evaluate hypergeometric series with
rational parameters where the argument $x$ is a high-precision number.
This reduces evaluating a $K$-term series to
$\OO(K)$ cheap scalar operations (additions and multiplications or divisions
by small integer coefficients)
and about $2\sqrt{K}$ expensive nonscalar operations (general multiplications),
whereas direct evaluation of the hypergeometric
recurrence uses $\OO(K)$ expensive operations.

\begin{algorithm}[h!]
  \caption{Evaluation of hypergeometric series using rectangular splitting}
  \small
  \label{alg:hyprs}
  \begin{algorithmic}[1]
    \Require An arbitrary $x$, recurrence data $p, q \in \mathbb{Z}[k]$, integer $K \ge 0$, offset $\Omega \in \{0,1\}$.
    \Ensure $s = \sum_{k=\Omega}^{K-1} x^k \prod_{j=\Omega}^k p(j) / q(j)$
    \State $m \gets \lfloor \sqrt K \rfloor$; precompute $[1, x, x^2, \ldots, x^m]$ \Comment{Tuning parameter: any $m \ge 1$ can be used}
    \State $s \gets 0; \quad k \gets K - 1$
    \While{$k \ge \Omega$}
        \State $u \gets \min(4, k + 1 - \Omega)$  \Comment{Tuning parameter: any $1 \le u \le k+1-\Omega$ can be used}
        \State $(a, b) \gets (k - u + 1, k)$  \Comment{Unrolled range}
        \State \label{step:coeff1} $c \gets \prod_{j=a}^b p(j)$ \Comment{Small integer coefficient}
        \While{$k \ge a$}
            \State $r \gets k \bmod m$
            \If{$k = b$}
                \State $s \gets c \cdot (s + x^r)$ \Comment{Using precomputed power of $x$}
            \Else
                \State $s \gets s + c \cdot x^r$ \Comment{Using precomputed power of $x$}
            \EndIf
            \If{$r = 0$ \textbf{and} $k \ne 0$}
                \State $s \gets s \cdot x^m$ \Comment{Using precomputed power of $x$}
            \EndIf
            \State \label{step:coeff2} $c \gets (c / p(k)) q(k)$ \Comment{Exact small integer division}
            \State $k \gets k - 1$
        \EndWhile
        \State $s \gets s / c$
    \EndWhile
    \State \Return $s$
  \end{algorithmic}
\end{algorithm}

\Cref{alg:hyprs} presents our version of rectangular splitting
for the present application.
We implement the various series expansions by defining
the functions $p(k), q(k)$ used in steps
\ref{step:coeff1}~and~\ref{step:coeff2} according to formulae
\cref{eq:asymprec}, \cref{eq:pzerorec},
\cref{eq:poneprec}, \cref{eq:ponerecprime}.

This algorithm is a generalization of the
method for evaluating Taylor series of elementary
functions given in~\cite{Johansson2015elementary},
which combines rectangular splitting
with partially unrolling the recurrence to reduce the number of scalar divisions
(which in practice are more costly than scalar multiplications).
The terms are computed in the reverse direction to allow
using Horner's rule for the outer
multiplications.

Our code uses ball arithmetic for $x$ and $s$
so that no error analysis is needed, and we use a bignum type for $c$ (so
no overflow can occur regardless of $u$).
For low precision a faster implementation
would be possible using fixed-point arithmetic with tight
control of the word-level operations
as was done for elementary functions in~\cite{Johansson2015elementary}.

The algorithm contains two tuning parameters%
\footnote{Let us stress that the choice of these parameters only affects
the performance of the algorithm, not its correctness.
The same remark holds every other time we resort to heuristics in this
article.}.
The splitting parameter $m$ controls the
number $m$ of multiplications for powers versus the number $K / m$ of
multiplications for Horner's rule.
The choice $m \approx \sqrt K$ is optimal,
but when evaluating two series for the same~$x$
(in our case, to compute both $P_n(x)$ and $P'_n(x)$),
the table of powers can be reused,
and then $m \approx \sqrt{2K}$ minimizes the total cost.

The unrolling parameter $u$ controls the number of coefficients
to collect on a single denominator, reducing the
number of divisions to $N / u$.
Ideally, $u$ should be chosen
so that $\prod_{j=a}^b p(j)$
and $\prod_{j=a}^b q(j)$ fit in 1 or 2 machine words.
The example value $u = 4$ is a reasonable
default, but as an optimization, one might vary $u$
for each iteration of the main loop
to ensure that $c$ always fits in a specific number of words.

The redundant parameter $\Omega$ is a small convenience in the pseudocode.
Setting $\Omega = 1$ and adding the constant term separately
avoids having to make a special case to prevent division by zero
when $q(0) = 0$.

Due to the scalar operations, rectangular splitting ultimately
requires $\OO(K)$ arithmetic operations with $\OOtilde(p)$ bit complexity
each
just like straightforward evaluation of the recurrence, so it is
not a genuine asymptotic improvement, but
it is an improvement in practice and
can give more than a factor 100 speedup at very high precision.
It is possible to genuinely reduce the complexity of evaluating
a hypergeometric sequence to $\OO(\sqrt{K} \log(K))$ arithmetic
operations using a baby-step giant-step method
that employs fast multipoint evaluation,
but in practice rectangular splitting
performs better until both $K$ and~$p$
exceed $10^6$ (see \cite{Johansson2014rectangular}).

\subsection{A note on the bit-burst method}

\label{sec:bitburst}

Another technique for fast evaluation of hypergeometric series,
called binary splitting, would be useful when $p$ is large and the
argument $x$ is a rational number with small numerator and denominator,
but this case is not relevant for our application.
Binary splitting also forms the basis of the bit-burst
method~\cite[Section 4]{ChudnovskyChudnovsky1990}, which permits
evaluating any fixed hypergeometric series at any fixed point---without
the restriction to simple rational numbers of plain binary
splitting---to absolute precision~$p$ in only~$\OOtilde(p)$ bit
operations.
Yet, we do not use this method either in our implementation, due to its
large overhead.

Indeed, computing~$P_n(x)$ by the bit-burst method method requires
$\OO(\log p)$ analytic continuation steps, each of which entails
two evaluations of general solutions of the Legendre differential
equation~\cref{eq:diffeq}.
These solutions are defined by unit initial values at some intermediate
point~$x_0$, and can be represented as a power series of radius of
convergence $1 - |x|$ whose coefficients obey recurrences of order two.
This is to be compared with a single series, given by a first-order
recurrence and typically converging faster, for the expansions
considered in \cref{sec:series-asy}~to~\cref{sec:series-one}.
Thus, the bit-burst method is unlikely to be competitive in the range
of precision we are interested in, especially when the asymptotic
series can be used.

Nevertheless, it can be shown that the solutions with unit initial
values at~$x_0$ of~\cref{eq:diffeq} occuring at intermediate analytic
continuation steps have Taylor coefficients~$c_k$ at~$x_0$ bounded by
$(n^2/(1-|x_0|))^{\OO(k)}$
uniformly in $n$~and~$x_0$.
As a consequence, the asymptotic cost of computing any individual root
of $P_n(x)$ by the bit-burst method and Newton iteration is
$\OOtilde(p) \, \log(n)^{\OO(1)}$.
For computing all the roots, this approach matches the $\OOtilde(np)$
estimate of \cref{thm:complexity}, while allowing for
parallelization and requiring less memory.
It may hence provide an alternative to multipoint
evaluation
worth investigating for precisions in the millions of bits.

\subsection{Binomial coefficients}

\label{sec:binomial}

The prefactors of both the series expansion at ${x = 0}$ and the asymptotic series
contain the central binomial coefficient ${2n \choose n}$.
We need to compute this factor efficiently for any $n$ and precision $p$.
Since ${2n \choose n} \approx 4^n$, it is best to use
an exact algorithm when $n < Cp$ for some small constant $C > 1/2$.
We use the binomial function provided by GMP for $n < 6p + 200$
and otherwise use an asymptotic series for ${2n \choose n}$
with error bounds given
in~\cite[Corollaries 1~and~2]{brent2016asymptotic}.

We also need to quickly estimate the magnitude of binomial coefficients
for error bounds of series truncations.
We have the binary entropy estimate
$$\log_2 {n \choose k} \le n G(k/n), \quad G(x) = -x \log_2(x) - (1-x) \log_2(1-x)$$
and the equivalent form
\begin{equation}
\label{eq:binbound}
{n \choose k} \le \left(\frac{n}{n-k}\right)^{n-k} \left(\frac{n}{k}\right)^k = \frac{n^n}{k^k (n-k)^{n-k}}.
\end{equation}
The function $G(x)$ can be evaluated cheaply with a precomputed
lookup table. A coarse estimate is sufficient, since overestimating
$\log_2 {n \choose k}$ by a few percent only adds a few
percent to the running time.

\section{Algorithm selection}

\label{sec:selection}

We first use a set of cutoffs found experimentally to decide whether to use the basecase recurrence or one of the series expansions. The recurrence is mainly faster for some combinations of $p < 1\,000$, $n < 400$ when computing $(P_n(x), P'_n(x))$ simultaneously and for some combinations of $p < 500$, $n < 100$ when computing $P_n(x)$ alone (in all cases subject to some boundaries $\varepsilon < x < 1 - \varepsilon$); the actual optimal regions are complicated due to differences in overhead between fixed-point integer arithmetic and ball arithmetic for the respective algorithm implementations. For the actual cutoffs used, we refer to the source code.

To select between the series expansion at $x = 0$, the expansion at $x = 1$, and the asymptotic series, the following heuristic is used. For each algorithm $A$, we estimate the evaluation cost as $C_A = K_A (p + p_A)$ where $K_A$ is the number of terms required by algorithm~$A$ ($K_A = \infty$ if $A$ is the asymptotic series and it does not converge to the required accuracy), $p$ is the precision goal, and $p_A$ is the extra precision required by algorithm $A$ due to internal cancellation. For the asymptotic series, we multiply the cost by an extra factor 2 as a penalty for using complex numbers. In the end, we select the algorithm with the lowest $C_A$.

We select $K_A$ and estimate $p_A$ heuristically using machine precision
floating-point computations, working with logarithmic magnitudes to
avoid underflow and overflow.
For example, when~$A$ is the asymptotic series, we search for a
small~$K_A$ such that $\log(K_A!/(2n\sin\theta)^{K_A}) \leq -(p + p_A)$
for some small~$p_A$, in accordance with~\cref{eq:truncerr0b}.
During the actual evaluation of the series expansions, $K_A$ and $p_A$ are then given;
we compute rigorous upper bounds
for the truncation error via
\cref{eq:truncerr0}, \cref{eq:truncerr0b}, \cref{eq:truncerr1},
\cref{eq:truncerr2}, \cref{eq:truncerr2b}, \cref{eq:binbound}
using floating-point arithmetic with directed rounding,
while additional rounding errors are tracked by the ball arithmetic.

\begin{figure}
\begin{centering}
\includegraphics[width=0.8\textwidth]{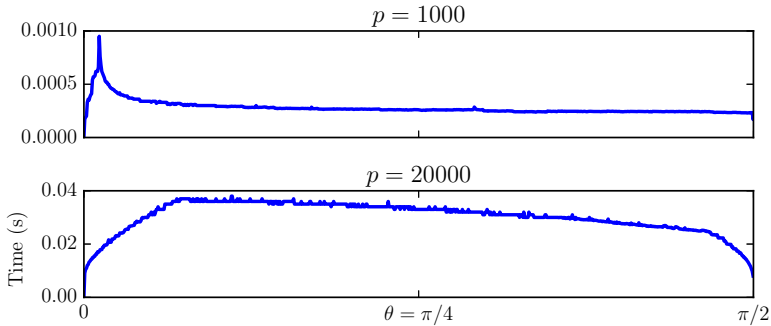}
\caption{Time to evaluate $(P_n(x), P'_n(x))$ as the argument $x = \cos(\theta)$ varies ($x = 1$ at $\theta = 0$ and $x = 0$ at $\theta = \pi/2$), here with $n = 10\,000$, for
precision $p$ somewhat smaller than $n$ (top plot) and somewhat larger (bottom plot). The variable $\theta$ is used for the horizontal scale in this picture
to follow the distribution of the roots (which are clustered near $x = 1$) linearly.}
\label{fig:timeplot}
\end{centering}
\end{figure}

The assumption that the running time is a bilinear function of $K_A$ and $p + p_A$
is not completely realistic, but this cost estimate nonetheless captures the
correct asymptotics when
$$x \to 0, \quad x \to 1, \quad n \to \infty, \quad p \to \infty$$
separately, and hopefully will not be too inaccurate in the transition regions.
This is verified empirically.

\Cref{fig:timeplot} illustrates how the time to evaluate
$(P_n(x), P'_n(x))$ varies with $x$ when the automatic algorithm selection is used.
Here, we have timed the case $n = 10\,000$ for two different $p$.
For large $n$ and $p \ll n$ (top plot), a sharp peak appears at the transition between the series
expansion at $x = 1$ and the asymptotic expansion which is used for most $x$.
This peak tends to become taller
but narrower for larger $n$.
We could presumably get rid of the peak by implementing another algorithm
specifically for the transition region, but the area under the peak is so small
compared to the median baseline that computing all the roots
would not be sped up much.
For $p$ somewhat larger than $n$ (bottom plot), we observe a smooth transition between the
series at $x = 1$ near the left of the picture and the series at $x = 0$
used over the most of the range.

\section{Benchmarks}

As already mentioned, our code for computing values and roots of
Legendre polynomials is part of the Arb library, available
from~\url{http://arblib.org/}.
The benchmark results given in this section were obtained using
prerelease builds of
version~2.13 of Arb.
The implementation of the algorithms of this article is located in the
files \texttt{arb\_hypgeom/legendre\_p\_ui\_*} of the Arb source tree.
In \cref{sec:bench-variants}, we also use some ad hoc code available
from this paper's public git repository%
\footnote{\url{https://github.com/fredrik-johansson/legendrepaper/}}
when comparing the main algorithm with variants absent from the Arb
implementation.
The source code for the experiments themselves can be found in the same
git repository.
Except where otherwise noted, we ran the programs under 64-bit Linux on a
laptop with a
1.90 GHz Intel Core i5-4300U CPU using a single core.

\label{sec:bench}

\subsection{Polynomial evaluation}

\label{sec:bench-variants}

\Cref{fig:benchplot} compares the performance of different methods
for evaluating $(P_n(x), P'_n(x))$ on a set of $n/2$ points
distributed like the positive roots of $P_n(x)$
to simulate one stage of Newton iteration at $p$-bit precision.
The time for the three-term recurrence is set to 1, i.e., we divide
the other timings by this measurement.
The following methods are timed:

\begin{itemize}
\item Our hybrid method (from here on called the ``main algorithm'')
with automatic selection between the three-term
recurrence and different
series expansions.
\item The main algorithm without
the three-term recurrence as the basecase, i.e., using series expansions
even for very small~$n$.
\item Fast multipoint
evaluation of the expanded polynomials $P_n(x)$ and $P'_n(x)$.
As in \cref{sec:series-zero}, we expand
$P_{n}(\sqrt{x})$ for even $n$ and $P_{n}(\sqrt{x})/\sqrt{x}$ for odd $n$
and evaluate at $x^2$ since this halves the amount of work.
The polynomial coefficients are generated using
the hypergeometric recurrence and
the fast multipoint evaluation is done using
\texttt{\_arb\_poly\_evaluate\_vec\_fast\_precomp}
(where the ``precomp'' suffix indicates that
the same product tree is used for both $P_n$ and $P'_n$).
The fast multipoint evaluation is done with $2.9n$ guard bits,
which was found experimentally to be sufficient for full accuracy.
\end{itemize}

\begin{figure}
\begin{centering}
\includegraphics[width=0.9\textwidth]{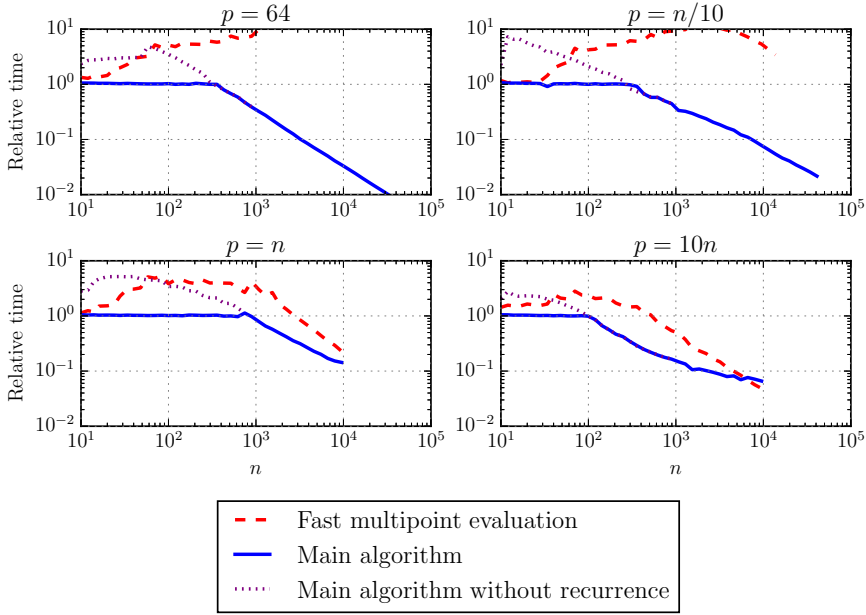}
\caption{Performance comparison of various methods to evaluate $(P_n(x), P'_n(x))$ to $p$-bit precision for a
set of $n / 2$ points $0 < x < 1$ distributed like the roots of $P_n$.
The $y$ axis (relative time) shows the time divided by the time
using the three-term recurrence in fixed-point arithmetic.}
\label{fig:benchplot}
\end{centering}
\end{figure}

The crossover point between the
three-term recurrence and series expansions usually occurs
around $n \approx 10^2 - 10^3$ (it can be as low as $n \approx 10$
if $p$ is much larger).
For modest $n$,
the three-term recurrence is much faster
than the hypergeometric series (typically by a factor 3-4)
due to working with negligible extra precision
and thanks to the low overhead of fixed-point arithmetic.
This low overhead is very useful for typical evaluation of
Legendre polynomials and generation of quadrature nodes
for one or a few machine words of precision.
The crossover point could be lowered slightly
if we used a similarly optimized fixed-point implementation
for the hypergeometric series.

When $p$ is fixed (top left in \cref{fig:benchplot}), the main algorithm is
a factor $\OO(n)$ faster
than the three-term recurrence
since the asymptotic expansion converges to sufficient accuracy after
$\OO(1)$ terms for all sufficiently large $n$.
With the constant precision $p = 64$,
the main algorithm is 3.0
times faster for $n = 10^3$ and 30 times faster for $n = 10^4$.
Conversely, fast multipoint evaluation
with constant $p$ is a factor $\OO(n)$ slower than our
algorithm due to the higher internal precision.

When $p \propto n$ (the three remaining plots in \cref{fig:benchplot}), the main algorithm
appears to show the same $\OO(n)$ speedup
over the three-term recurrence after the crossover point, at least initially.
This speedup should level off asymptotically, but
in practice this only occurs for $n$ larger than $10^4$
where we have already gained a factor 10 or more.
The leveling off is visible in the bottom right figure ($p = 10n$).

Fast multipoint evaluation gives a true asymptotic $\OO(n)$ speedup,
but since it has much higher overhead,
it only starts to give an improvement over
the main algorithm from $n \approx 10^4$
and for $p$ larger than $n$.
When $p = n / 10$, it appears that
fast multipoint evaluation will only break even for $n$ much larger than $10^5$.
We conclude that fast multipoint evaluation would be
worthwhile only for the last few Newton iterations
when computing quadrature nodes for exceptionally high precision.
Since independent evaluations are more convenient and easy to parallelize,
the fast multipoint evaluation method
currently seems to have limited practical value for this application.

\subsection{Quadrature nodes}

\begin{table}[t!]
\caption{Time in seconds to compute the degree-$n$ Gauss-Legendre rules with $p$-bit precision
using our code.
(For large $n$ and $p$, the time was estimated by computing a subset of the nodes and weights.)}
\label{tab:timings}
\begin{center}
\begin{tabular}{ r | c c c c c }
$n\, \backslash \; p$ & 64 & 256 & 1\,024 & 3\,333 & 33\,333 \\
\hline\rule{0pt}{3ex}
20  & 0.000133  &  0.000229  &  0.000510  &  0.00121  &  0.0198  \\
50  & 0.000450  &  0.000870  &  0.00212  &  0.00520  &  0.0710  \\
100  & 0.00138  &  0.00310  &  0.00720  &  0.0163  &  0.191  \\
200  & 0.00550  &  0.0111  &  0.0267  &  0.0550  &  0.589  \\
500  & 0.0236  &  0.0610  &  0.164  &  0.325  &  2.61  \\
1\,000  & 0.0530  &  0.145  &  0.584  &  1.238  &  9.21  \\
2\,000  & 0.0860  &  0.298  &  1.12  &  4.20  &  32.6  \\
5\,000  & 0.191  &  0.665  &  2.67  &  14.3  &  181  \\
10\,000  & 0.350  &  1.26  &  4.93  &  26.6  &  674  \\
100\,000  & 3.60  &  12.2  &  41.3  &  212  &  13\,637  \\
1\,000\,000  & 58.0  &  146  &  411  &  1\,850  &  103\,960  \\
\end{tabular}
\end{center}
\end{table}


\begin{table}[t!]
\caption{Time in seconds for Pari/GP to compute the degree-$n$ Gauss-Legendre quadrature rules with $p$-bit precision. The
numbers in parentheses show the speedup of our code compared to Pari/GP.}
\label{tab:parispeedup}
\begin{center}
\setlength{\tabcolsep}{1ex}
\begin{tabular}{ r | c c c c c }
$n\, \backslash \; p$ & 64 & 256 & 1\,024 & 3\,333 & 33\,333 \\ \hline\rule{0pt}{3ex}
20      & 0.00059 (\footnotesize{$\times 4.4$)}  &  0.00070 (\footnotesize{$\times 3.1$)}  &  0.0015 (\footnotesize{$\times 2.9$)}  &  0.0035 (\footnotesize{$\times 2.9$)}  &  0.078 (\footnotesize{$\times 3.9$)}  \\
50      & 0.0043 (\footnotesize{$\times 9.6$)}  &  0.0050 (\footnotesize{$\times 5.8$)} &  0.010 (\footnotesize{$\times 4.9$)}  &  0.022 (\footnotesize{$\times 4.1$)} &  0.43 (\footnotesize{$\times 6.0$)}  \\
100     &  0.020 (\footnotesize{$\times 15$)}   &   0.023 (\footnotesize{$\times 7.4$)}  &   0.045 (\footnotesize{$\times 6.3$)}  &  0.089 (\footnotesize{$\times 5.5$)} &  1.5 (\footnotesize{$\times 8.1$)}  \\
200     &  0.11 (\footnotesize{$\times 20$)}   &  0.13 (\footnotesize{$\times 11$)}  &  0.24 (\footnotesize{$\times 9.1$)}  &  0.45 (\footnotesize{$\times 8.1$)}  &  6.5 (\footnotesize{$\times 11$)}  \\
500     &  2.0 (\footnotesize{$\times 86$)}     & 2.1 (\footnotesize{$\times 35$)}  &  2.9 (\footnotesize{$\times 18$)}  &  5.1 (\footnotesize{$\times 16$)}  &  58 (\footnotesize{$\times 22$)}  \\
1\,000  & 25 (\footnotesize{$\times 477$)}    &  26 (\footnotesize{$\times 180$)} &  29 (\footnotesize{$\times 49$)}  &  39 (\footnotesize{$\times 32$)}  & 300 (\footnotesize{$\times 33$)} \\
2\,000  &  496 (\footnotesize{$\times 5\,767$)}   & 478 (\footnotesize{$\times 1\,604$)}  &  474 (\footnotesize{$\times 423$)}  &  532 (\footnotesize{$\times 127$)}  &  1\,880 (\footnotesize{$\times 58$)}  \\
\end{tabular}
\end{center}
\end{table}

The function
\texttt{arb\_hypgeom\_legendre\_p\_ui\_root(x, w, n, k, p)}
sets the output variable
$x$ to a ball containing the root of $P_n$ with index~$k$ (we use the indexing
$0 \le k < n$, with $k = 0$ giving the root closest to 1),
computed to $p$-bit precision.
It also sets $w$ to the corresponding quadrature weight.
We use the formulae in \cite[Theorem~1(c)]{petras1999computation} to
compute an initial enclosure with roughly machine precision,
followed by refinements with the interval Newton method
at doubling precision steps for very high precision.
We deduce the quadrature weights thanks to the classical expression as
functions of the nodes recalled in equation~\cref{eq:glquad}.

\Cref{tab:timings} shows timings for computing
degree-$n$ Gauss-Legendre rules to $p$-bit precision
by calling this function repeatedly with $0 \le k < n / 2$.
\Cref{tab:parispeedup} compares our code
to the \texttt{intnumgaussinit} function in Pari/GP
which uses a generic polynomial root isolation strategy
followed by Newton iteration for high precision refinement.
The improvement is most dramatic for small~$p$ and large~$n$ where
we benefit from using asymptotic expansions, but we also obtain a
consistent speedup for large $p$.

For low precision and large $n$, our implementation
is about three orders of magnitude slower than the machine precision
code by Bogaert \cite{bogaert2014iteration}
which is reported to compute the nodes and weights for $n = 10^6$
in 0.02 seconds on four cores.
This difference is reasonable since we use arbitrary-precision arithmetic,
compute rigorous error bounds, and evaluate the Legendre polynomials
explicitly whereas Bogaert uses a more sophisticated
asymptotic development for both the nodes and the weights.

We also note that we can compute 53-bit floating-point values
with provably correct rounding in about the same time as the 64-bit
values, using Ziv's strategy of increasing the precision.
For a ball with relative radius just larger than $2^{-64}$, there
is less than a $1\%$ probability that the correct 53-bit rounding
cannot be determined, in which case that particular node
can be recomputed with a few more bits.

Fousse~\cite{fousse2007accurate} reports a few timings for smaller $n$ and high precision
obtained on a 2.40 GHz AMD Opteron 250 CPU.
For example, $n = 80, p = 500$ takes 0.14 seconds (our implementation
takes 0.029 seconds) and $n = 556, p = 5\,000$ takes 17 seconds
(our implementation takes 0.53 seconds).
Of course, these timings are not directly comparable since different
CPUs were used.

The \texttt{mathinit} program included with version 2.2.19
of D.\ H.\ Bailey's ARPREC library
generates Gauss-Legendre quadrature nodes
using Newton iteration together with the three-term recurrence for evaluating
Legendre polynomials~\cite{bailey2002arprec,bailey2011high}.
With default parameters, this program computes the rules of degree $n = 3 \cdot 2^{i+1}$
for $1 \le i \le 10$ at 3\,408 bits of precision, intended as
a precomputation for performing
degree-adaptive numerical integrations with up to 1\,000 decimal digit accuracy.
This takes about 1\,300 seconds in total (our implementation takes 32
seconds).
A breakdown for each degree level
is shown in \cref{tab:arprectimings}.

\Cref{tab:arprectimings} also shows the approximation error and the
evaluation time (not counting the computation of the nodes and weights)
for the degree-$n$ approximations of three different integrals,
illustrating the relative costs and realistic requirements for $n$.
As motivation for the third integral, we might think
of a segment of a Mellin-Barnes integral.
The log, Airy and gamma function implementations in Arb are used.

\begin{table}[t!]
\caption{Left columns: time in seconds to generate 1\,000-digit
quadrature rules for the degrees~$n$ used by ARPREC.
Right columns: for three different integrals, the
error $|\int_{-1}^1 f(x) \dx - \sum_{k=0}^{n-1} w_k f(x_k)|$ of the degree-$n$
quadrature rule, and the time to evaluate this degree-$n$ approximation
of the integral at 1\,000-digit precision in Arb given the nodes and weights $(x_k, w_k)$.}
\begin{center}
\begin{tabular}{ r | c c | c c| c c | c c }
$n$ & ARPREC\! & Our code &
    \multicolumn{2}{|c|}{$\int_{-1}^{1}\!\log(2\!+\!x) \dx$} &
    \multicolumn{2}{|c|}{$\int_{-1}^{1}\!\operatorname{Ai}(10 x) \dx$} &
    \multicolumn{2}{|c}{$\int_{-1}^{1}\!\Gamma(1\!+\!ix) \dx$} \\
   &         &          & Error   & Time      &  Error & Time  &  Error & Time \\ \hline
\rule{0pt}{3ex}12 & 0.00520 & 0.000592 & $10^{-14}$ &       &   $10^{-1}$ &  &  $10^{-8}$ & \\
24 & 0.0189 & 0.00171 & $10^{-28}$  &       &  $10^{-9}$  & &  $10^{-17}$ & \\
48 & 0.0629 & 0.00507 & $10^{-56}$ &        &  $10^{-34}$  &    &  $10^{-36}$ & \\
96 & 0.251 & 0.0163 & $10^{-111}$  &         &  $10^{-105}$ &          &  $10^{-73}$ & \\
192 & 0.974 & 0.0532 & $10^{-222}$ &         &  $10^{-284}$ &    0.075      &   $10^{-146}$ & \\
384 & 3.83 & 0.195 & $10^{-441}$  &  0.023        & $10^{-721}$ & 0.15           &   $10^{-293}$ & 1.3 \\
768 & 15.2 & 0.763 & $10^{-881}$ & 0.045      & $<\varepsilon$ & 0.29           &   $10^{-588}$ & 2.5 \\
1\,536 & 60.9 & 2.82 & $<\varepsilon$ &  0.091     &                &    &            $<\varepsilon$ & 5.0 \\
3\,072 & 241 & 9.55 &  &  &  &  &  &  \\
6\,144 & 1\,013 & 18.3 &  &  &  &  &  &  \\
\end{tabular}
\label{tab:arprectimings}
\end{center}
\end{table}

\begin{table}[t!]
\caption{Left columns: step sizes $h$, number of evaluation
points, and time to compute nodes for double exponential quadrature with Arb
at 1\,000-digit precision. Right columns: error and evaluation
time given precomputed nodes.}
\begin{center}
\begin{tabular}{ c c c | c c | c c | c c }
$h$ & $2n+1$ & Time &
    \multicolumn{2}{|c|}{$\int_{-1}^{1}\!\log(2\!+\!x) \dx$} &
    \multicolumn{2}{|c|}{$\int_{-1}^{1}\!\operatorname{Ai}(10 x) \dx$} &
    \multicolumn{2}{|c}{$\int_{-1}^{1}\!\Gamma(1\!+\!ix) \dx$} \\
   &         &          & Error   & Time      &  Error & Time  &  Error & Time \\ \hline
\rule{0pt}{3ex}$2^{-7}$ & 1\,989 & $0.07$ & $10^{-407}$ & 0.12 & $10^{-423}$ & 0.93 & $10^{-314}$ & 6.3 \\
$2^{-8}$ & 3\,977 & $0.14$ & $10^{-814}$ & 0.25 & $10^{-909}$ & 1.75 & $10^{-630}$ & 13.0 \\
$2^{-9}$ & 7\,955 & $0.27$ & $<\varepsilon$ & 0.55 & $<\varepsilon$ & 3.49 & $<\varepsilon$ & 25.1
\end{tabular}
\label{tab:dequad}
\end{center}
\end{table}

The last few degree levels (with $n$ roughly larger than
the number of decimal digits) used by ARPREC
tend to be dispensable for well-behaved integrands.
A larger $n$ is needed if the path of integration
is close to a singularity or if the integrand is highly oscillatory.
In such cases, bisecting the interval a few times
to reduce the necessary $n$ is often a better tradeoff.
On the other hand, since the time to generate nodes with our code only
grows linearly with~$n$ beyond $n \approx p$,
increasing the degree further is viable,
and potentially useful if the integrand is expensive to evaluate.

In the present work, we refrain from a more detailed discussion of
adaptive integration strategies and the computation
of error bounds for the integral itself.
However, we mention that Arb contains an implementation of a
version of the Petras algorithm~\cite{petras2002self} for rigorous
integration. This code uses both adaptive path subdivision
and Gauss-Legendre quadrature with an adaptive choice of $n$ up to
$n \approx 0.5p$ by default, with degree increments
$n \approx 2^{k/2}$ and automatic caching of the nodes
for fast repeated integrations.
Node generation takes at most a few seconds for a first integration
at 1\,000-digit precision
and a few milliseconds for 100-digit precision.

\section{Gauss-Legendre versus Clenshaw-Curtis and the double exponential method}

\label{sec:vsothers}

The Clenshaw-Curtis and double exponential (tanh-sinh) quadrature schemes
have received much attention as alternatives to Gauss-Legendre
quadrature for numerical integration with very high precision~\cite{takahasi1974double,bailey2011high,trefethen2008gauss}.
Both schemes
typically require a constant factor more evaluation points than Gauss-Legendre rules
for equivalent accuracy, but the nodes and weights are easier to compute.
Gauss-Legendre quadrature is therefore the best choice
when the integrand is expensive to evaluate or when nodes can
be precomputed for several integrations.
It is of some interest to compare the relative costs empirically.

Here we assume an analytic integrand with singularities well isolated from
the finite path of integration so that Gauss-Legendre quadrature is a good
choice to begin with. As observed in~\cite{trefethen2008gauss},
Clenshaw-Curtis often converges with identical rate to Gauss-Legendre for
less well-behaved integrands,
and the double exponential method is far superior to either
Clenshaw-Curtis or Gauss-Legendre for
analytic integrands with endpoint singularities.

Clenshaw-Curtis quadrature uses the Chebyshev nodes
$\cos(\pi k / n), 0 \le k \le n$, and the corresponding weights can be expressed
by a discrete cosine transform which takes $\OO(n \log n)$
arithmetic operations to compute by an FFT.
As a rule of thumb, $2n$-point Clenshaw-Curtis quadrature
gives the same accuracy as $n$-point Gauss-Legendre quadrature
(for instance, the 384-point Clenshaw-Curtis rule gives
errors of $10^{-229}$, $10^{-294}$ and $10^{-154}$ for the three
integrals in \cref{tab:arprectimings}).
As a point of comparison with \cref{tab:timings,tab:arprectimings},
Arb computes a length-2\,048 FFT with 1\,000-digit precision
in 0.09 seconds and a length-32\,768 FFT with 10\,000-digit precision
in 36 seconds.
The precomputation for Clenshaw-Curtis quadrature is therefore
roughly a factor 20 cheaper than for Gauss-Legendre quadrature
with our algorithm, while subsequent integration with
cached weights is twice as expensive for Clenshaw-Curtis.

Double exponential quadrature uses the change of variables
$x = \tanh(\tfrac12 \pi \sinh t)$ to convert an integral on $(-1,1)$
to the interval $(-\infty,+\infty)$ in such a way that
the trapezoidal rule $\int_{-\infty}^{\infty} f(t) dt \approx h \sum_{k=-n}^{n} f(hk)$
converges exponentially fast.
One generally chooses the discretization parameter as $h = 2^{-j}$ so that
both the evaluation points and weights can be recycled for successive levels
$j = 1, 2, 3\ldots$, and $n$ is chosen so that the tail of the
infinite series is smaller than $2^{-p}$.
The $2n+1$ nodes and weights can be computed with $n + \OO(1)$
exponential function evaluations and $\OO(n)$ arithmetic operations.
Double exponential quadrature with $Cn$ evaluation points
typically achieves the same accuracy as $n$-point Gauss-Legendre quadrature,
where $C$ is slightly larger than
for Clenshaw-Curtis, e.g.\ $C \approx 5$; see \cref{tab:dequad}.
The time to compute nodes and weights is comparable
to Clenshaw-Curtis quadrature (around 0.2 seconds for 1\,000-digit
precision and two minutes for 10\,000-digit precision).

In summary, for integration
with precision in the neighborhood of $10^3$ to $10^4$ digits,
computing $n$ nodes with our algorithm is about an
order of magnitude more expensive than performing $n$ elementary function (e.g.\ exp or log)
evaluations or computing an FFT of length~$n$.
This makes Gauss-Legendre quadrature competitive for computing more than $m$
integrals (assuming that nodes and weights are cached),
for a single integral requiring splitting into $m$ subintervals,
or for a single integral when the integrand costs more than $m$ elementary
function evaluations, where $m \approx 10^1$.

The picture becomes more complicated when accounting for the method
used to estimate errors in an adaptive integration algorithm.
One drawback of the Gauss-Legendre scheme is that the nodes
are not nested, so that an adaptive strategy that
repeatedly doubles the quadrature degree
requires twice as many function evaluations as
the degree of the final level.
This drawback disappears if the
error is estimated by extrapolation or if an error bound is computed
a priori as in the Petras algorithm~\cite{petras2002self}.

\section{Conclusion}

In \cite{bailey2011high}, it was claimed that
``There is no known scheme for generating Gaussian abscissa--weight pairs
that avoids [the] quadratic dependence
on~$n$. High-precision abscissas
and weights, once computed, may be stored for future use. But for truly
extreme-precision calculations -- i.e., several thousand digits or
more -- the cost of computing them
even once becomes prohibitive''.

In this quote, ``quadratic dependence'' refers to the number of arithmetic operations.
We may remark that using asymptotic expansions in the
evaluation of Legendre polynomials avoids the quadratic
dependence on~$n$ for fixed precision $p$, and \cref{thm:complexity}
avoids the implied cubic time dependence on $n$ when $n = \OO(p)$.
In fact, \cref{thm:complexity} implies that
Gauss-Legendre, Clenshaw-Curtis and double exponential
quadrature have the same quasi-optimal asymptotic bit complexity,
up to logarithmic factors.

Our experiments show that the algorithm in \cref{thm:complexity}
hardly is worthwhile. However, the hybrid method described in
\crefrange{sec:general}{sec:eval} does achieve a
significant speedup for practical $p$ and $n$
which allows us to compute Gauss-Legendre quadrature rules for
1\,000-digit integration in 1-2 seconds and for 10\,000-digit integration in 10-20 minutes
on a single core. This is not prohibitively
expensive compared to the repeated evaluation
of typical integrands, especially if several integrations are needed.
Parallelization is also trivial since all roots
are computed independently.

A natural extension of this work would be to
consider Gaussian quadrature rules for different
weight functions.
The techniques should transfer to other
classical orthogonal polynomials (Jacobi, Hermite, Laguerre, etc.)
which likewise have hypergeometric expansions
and satisfy three-term recurrence relations.
The main obstacle might be
to obtain large-$n$ asymptotic expansions with suitable error bounds.

\section*{Acknowledgements}

We are indebted to Nick Trefethen and two anonymous referees for their
useful comments. We also thank Bruno Salvy for pointing out Petras'
paper~\cite{petras1999computation}.
Marc Mezzarobba was supported in part by ANR grant ANR-14-CE25-0018-01
(Fast\-Relax).


\end{document}